\documentclass{birkjour}
\usepackage{subfigure}
\usepackage{verbatim}
 \newtheorem{thm}{Theorem}[section]
 \newtheorem{corollary}[thm]{Corollary}
 \newtheorem{lemma}[thm]{Lemma}
 \newtheorem{proposition}[thm]{Proposition}
 \theoremstyle{definition}
 
 \theoremstyle{remark}
 \newtheorem{remark}[thm]{Remark}
 \newtheorem{example}{Example}
 \numberwithin{equation}{section}
 
 \newcommand{\R}{\mathbb{R}}

\begin{document}

%
%
%
%
%
%
%
%
%

\title[Convex Improper Affine Maps]
 {Singularities of Convex Improper Affine Maps}

\author[M.Craizer]{Marcos Craizer}

\address{%
Departamento de Matem\'{a}tica- PUC-Rio\br
Rio de Janeiro\br
Brazil}

\email{craizer@puc-rio.br}


\subjclass{ 53A15}

\keywords{ Improper affine spheres, Generating family}

\date{July 13, 2010}

\begin{abstract}
In this paper we consider convex improper affine maps of the $3$-dimensional affine space
and classify their singularities. The main tool developed is a generating family with properties that
closely resembles the area function for non-convex improper affine maps. 
\end{abstract}

\maketitle

\section{Introduction}

Improper affine spheres are surfaces in $\R^3$ whose affine normal vector field is constant. 
This type of surfaces has many interesting properties and is being studied for a long time (\cite{Calabi58},\cite{Calabi70}). Improper affine maps are improper affine spheres with some admissible singularities.
They were first considered in  \cite{Martinez05} in the convex case and in  \cite{Nakajo09} in the non-convex case. 

Every non-convex improper affine map can be obtained from a pair of planar curves. Denoting by $m$ the mid-point of a chord 
and by $f$ the area of the planar region bounded by this chord, the two curves and another fixed chord, 
$(m,f)$ is a non-convex improper affine map
whose singular set projects into the Wigner caustic of the pair of  curves (\cite{Rios10}).
The singular set of a non-convex improper affine maps was studied in  \cite{Craizer11}, where it is proved that
generically it consists of cuspidal edges and swallowtails. 

In this work we are concerned with the convex case.   In \cite{Martinez05}, it is proved that any convex improper affine sphere can be obtained by the following construction:
Consider a planar map 
$x(s,t)=(A(s,t),B(s,t))$ whose coordinates $A$ and $B$ are harmonic functions. Let $C(s,t)$ and $D(s,t)$ be conjugate harmonics of $B(s,t)$ and $-A(s,t)$
in such a way that $B+iC$ and $A-iD$ become holomorphic. Denote by $g(s,t)$ a function whose gradient with respect to $x$ 
is $(C,D)$. Then $q=(x,g)$ is a convex improper affine sphere. Admitting singularities, the map $q$ is called a {\it convex improper affine map}. 
 
We shall study the singular set of a convex improper affine map. The singular set of the planar map $x$ 
is well-known: Generically it admits regular singular points and cusps. 
The main result of this paper is that the corresponding 
singular points of the improper affine map are cuspidal edges and swallowtails. 

The main tool used in proving our results is the construction of a generating family. In the non-convex case, the generating family is the area function and 
was studied in \cite{Giblin08}. The generating family that we construct here for the convex case has not an obvious geometric meaning, but was inspired
by the area function of the non-convex case.

The paper is organized as follows: In section 2 we recall the representation formula for convex improper affine maps. In section 3
we discuss the well-known results concerning singularities of the planar map. Section 4 contains the main calculations of the paper:
we define the generating family and prove many of its properties. In section 5 we classify the singularities of the convex improper affine maps.

\paragraph*{Acknowledgements:} The author wants to thank CNPq for financial support during the preparation of the paper.

\section{Improper affine maps}

Consider a pair of  holomorphic maps $(s,t)\to A(s,t)-iD(s,t)$ and $(s,t)\to B(s,t)+iC(s,t)$, for $(s,t)$ in some simply connected domain $\Omega\subset\R^2$. Thus we have 
$A_s=-D_t$, $A_t=D_s$, $B_s=C_t$ and $B_t=-C_s$. 
Denote by 
$$
x(s,t)=(A(s,t), B(s,t))
$$
the planar map and by
$$
c(s,t)=(-D(s,t), C(s,t))
$$
the conjugate map. 
The derivatives $x_s=(A_s,B_s),x_t=(A_t,B_t)$, $c_s=(-D_s,C_s),c_t=(-D_t,C_t)$ satisfy the conditions $x_s=c_t$ and $x_t=-c_s$. For any vectors $u,v\in\R^2$, denote by $\left[ u,v\right]$
the determinant of the matrix whose columns are $u$ and $v$. Let
$$
\delta(s,t)=[x_s,x_t]=A_sB_t-A_tB_s.
$$

\begin{lemma}
There exists a function $g:\Omega\to\R$ whose symplectic gradient with respect to $x$ is $c$, i.e., 
\begin{equation}\label{defineg}
\left\{
\begin{array}{ll}
g_s=\left[ x_s, c  \right]\\
g_t=\left[ x_t, c  \right]
\end{array}
\right.
\end{equation}
The function $g$ is uniquely defined by $x$ and $c$ up to an additive constant. At points $(s,t)$ such that $\delta(s,t)\neq 0$, the map $q=(x,g)$ is 
an improper affine sphere. Admitting singularities, the map $q=(x,g)$ is an improper affine map.
\end{lemma}
\begin{proof}
The function $g$ must satisfy \eqref{defineg}. Thus 
$(g_s)_t=\left[ x_{st}, c \right]=(g_t)_s.$
Since $\Omega$ is simply connected, the existence of $g$ is proved. Now
\[
\left\{
\begin{array}{l}
q_s=\left( x_s, [x_s, c] \right)\\
q_t=\left( x_t, [x_t,c] \right).
\end{array}
\right.
\]
Considering the transversal vector field $\xi=(0,0,1)$, we shall decompose the second derivatives of $q$
in the basis $\{q_s, q_t,\xi\}$. We have
\[
\left\{
\begin{array}{l}
q_{ss}=(x_{ss}, [x_{ss},c])-(0,[x_s,x_t])\\
q_{tt}=(x_{tt}, [x_{tt},c])-(0,[x_s,x_t])\\
q_{st}=(x_{st}, [x_{st},c]).
\end{array}
\right.
\]
The coefficients of $\xi$ in this decomposition determine a metric 
\[
h=-\left[
\begin{array}{cc}
[x_s,x_t] & 0 \\
0 & [x_s,x_t]
\end{array}
\right] 
\]
such that the corresponding volume form is $\omega_h=\sqrt{|det(h)|}=|[x_s,x_t]|$. Since this volume form coincides with
$[q_s,q_t,\xi]$, we conclude that $\xi$ is in fact the Blaschke affine normal vector, and so $q=(x,g)$ is an improper affine sphere. 
\end{proof}

\begin{example}\label{example1}
Let $A(s,t)=s+\frac{s^2-t^2}{2}$, $B(s,t)=st-t$, $C(s,t)=\frac{t^2-s^2}{2}+s$ and $D(s,t)=-st-t$. Then 
\[
\begin{array}{c}
g_s=s+\frac{s^2-t^2}{2}-\frac{s(s^2+t^2)}{2}\\
g_t=t-ts-\frac{t^3}{2}-\frac{ts^2}{2}
\end{array}
\]
and thus
$$
g(s,t)=\frac{s^2+t^2}{2}+\frac{s^3}{6}-\frac{s^4+t^4}{8}-\frac{st^2}{2}-\frac{s^2t^2}{4}.
$$
\end{example}

\begin{remark}\label{remark:addconstant}
If we add a constant $(k_1,k_2)$ to $(C,D)$, then the new $g$ will be obtained from the original by the sum of $k_1A(s,t)+k_2B(s,t)$. 
\end{remark}

\begin{remark}
In \cite{Martinez05}, the following representation formula for improper affine spheres with definite metric is proved:
Given two holomorphic functions $F(z)=F_1(z)+iF_2(z)$ and $G(z)=G_1+iG_2(z)$, write 
\[
\left\{
\begin{array}{ll}
x=\overline{F}+G=\left(F_1+G_1,G_2-F_2\right)\\
n=\overline{F}-G=\left(F_1-G_1,-G_2-F_2\right)
\end{array}
\right.
\]
We can pass from this formula to the one described above by taking, 
for $z=s+it$, $A(s,t)=F_1+G_1$, $B(s,t)=-F_2+G_2$, $C(s,t)=F_1-G_1$, $D(s,t)=-G_2-F_2$. 
\end{remark}

\section{Singularities of the planar map}

The singular set $S$ is defined by the equation $\delta(s,t)=0$. We shall assume
that  $(\delta_s,\delta_t)\neq(0,0)$ so that $S$ is a regular curve. This condition imply in particular that the jacobian matrix of $x$ 
has rank $1$, for any $(s,t)\in S$.

Fix a particular $(s,t)\in S$. By a rotation in the $(s,t)$-plane we may assume that $\delta_s(s,t)\neq 0$ and $\delta_t(s,t)\neq 0$. 
Thus, in a neighborhood of $(s,t)$, we may parameterize $S$ by $s$ or $t$. Since $(x_s(s,t),x_t(s,t))\neq (0,0)$, we shall assume
without loss of generality that $x_t(s,t)\neq 0$. Under this assumption, we choose $s$ as a parameter for the singularity set $S$.

Consider a parameterization $(s,\tau(s))$ of $S$ and 
define $\alpha(s)$ by 
\begin{equation}\label{definealpha}
x_s(s,\tau(s))+\alpha(s)x_t(s,\tau(s))=0
\end{equation}
at points of $S$.

\begin{thm}\label{thm:singularplanar} The following holds:
\begin{enumerate}
\item
If $\alpha(s)\neq\tau_s(s)$  then $x(S)$ is smooth at $x(s,\tau(s))$.
\item
If $\alpha(s)=\tau_s(s)$ but $\alpha'(s)\neq\tau_{ss}(s)$, then $x(S)$ has an ordinary cusp at $x(s,\tau(s))$.
\end{enumerate}
\end{thm}


We refer to \cite{Whitney55}  for this well-known theorem. Nevertheless, in section \ref{section:generating} we shall give another proof of this result.

\begin{example}\label{example2}
Taking $A(s,t)=s+\frac{s^2-t^2}{2}$ and $B(s,t)=st-t$ as in example \ref{example1}, we get $\delta=s^2+t^2-1$. Thus the singularity set $S$ is the unit circle. For $t\neq 0$, $\tau_s=-\frac{s}{t}$ and $\alpha(s)=\frac{1+s}{t}$.
Thus $\alpha(s)=\tau_s(s)$ only for $s=-\frac{1}{2}$, $t=\pm\frac{\sqrt{3}}{2}$, and hence $x(S)$ is regular for $s\neq -\frac{1}{2}$, $t\neq 0$. 
For $s=-\frac{1}{2}$, $\alpha'(s)\neq\tau_{ss}$ and so $x(S)$ has an ordinary cusp at these points.

For $t=0$, we cannot use $s$ as a parameter for $S$, and so we use $t$. Then $\alpha(t)=0$ if and only if $s=1$. Thus $x(S)$ is regular for $s=-1$.
For $s=1$, $\alpha'(t)\neq\tau_{tt}$ and so $x(S)$ has an ordinary cusp at this point (see figure \ref{trescuspides}). 

\begin{figure}[htb]
 \centering
 \includegraphics[width=0.25\linewidth]{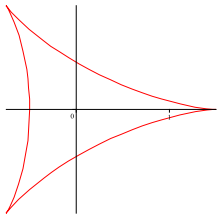}
 \caption{The image of the singularity set of example \ref{example2} by the planar map.}
\label{trescuspides}
\end{figure}
\end{example}

\section{Generating family}\label{section:generating}

We shall study the behavior of the map $q=(x,g)$ near a singular point $(s_0,t_0)$. As explained above, we may assume, without loss of generality, that $x_t(s_0,t_0)\neq 0$. 
By remark \ref{remark:addconstant}, we may add a constant $(k_1,k_2)$ to $c(s,t)$ and the new map $g$ will be obtained from the original by adding a smooth function, and thus
the type of singularity remains the same. Thus we may assume that 
\begin{equation}\label{hypothesis1}
\left[ x_t(s_0,t_0), c(s_0,t_0)\right] \neq 0.
\end{equation}
From now on we shall assume that \eqref{hypothesis1} holds and $\delta_t\neq 0$. 

In what follows we shall use the symbol $x$ with a double meaning: as a function of $(s,t)$ and as a variable in itself. We hope that this will not cause any confusion to the reader. 

\begin{lemma}
We can define a function $t=t(s,x)$ in a neighborhood $I\times J\times U$ of $(s_0,t_0,x_0)$, $x_0=x(s_0,t_0)$,
implicitly by the relation 
\begin{equation}\label{definet}
\left[ x-x(s,t), c(s,t) \right]=0.
\end{equation}
\end{lemma}

\begin{proof} 
Let $I_0\times J_0\subset E$ be a neighborhood of $(s_0,t_0)$ and $U_0\subset\R^2$ be a neighborhood of $x_0$. 
Define $F:I_0\times J_0\times U_0\to\R$ by 
$$
F(s,t,x)=\left[ x-x(s,t), c(s,t)    \right].
$$
Then 
$$
\frac{\partial F}{\partial t}=-[x_t(s,t), c(s,t)]+[x-x(s,t),c_t(s,t)] 
$$
is not zero at $(s_0,t_0,x_0)$. By the implicit function theorem, the equation $F(s,t,x)=0$ define $t$ as a function of $(s,x)$
in a neighborhood $I\times J\times U$ of $(s_0,t_0,x_0)$, where $I\subset I_0$ is a neighborhood of $s_0$, $J\subset J_0$ is a neighborhood of $t_0$ and $U\subset U_0$ is a neighborhood
of $x_0$.
\end{proof}

Define the generating family $G:I\times U\to\R$ by 
\begin{equation}\label{defineG}
G(s,x)=g(s,t(s,x)),
\end{equation}
where $t=t(s,x)$ is defined implicitly by equation \eqref{definet}. 
The goal of this section is to prove the following theorem:

\begin{thm} \label{thm:familyG}  
Take $(s,x)\in I\times U$:
\begin{enumerate}
\item 
$(s,x)$ is an $A_1$ point for $G$  if and only if $x=x(s,t)$ and $\delta(s,t(s,x))\neq 0$.
\item
$(s,x)$ is an $A_2$ point for $G$  if and only if $x=x(s,t)$, $\delta(s,t(s,x))=0$ and $\frac{\partial}{\partial s}\delta(s,t(s,x))\neq 0$, or equivalently, 
$\alpha(s)\neq\tau_s(s)$. Moreover, the unfolding $G$ is versal at this point. 
\item
$(s,x)$ is an $A_3$ point for $G$  if and only if $x=x(s,t)$, $\delta(s,t(s,x))=0$, $\frac{\partial}{\partial s}\delta(s,t(s,x))=0$ and $\frac{\partial^2}{\partial s^2}\delta(s,t(s,x))\neq 0$, or equivalently, $\alpha(s)=\tau_s(s)$ and $\alpha'(s)\neq\tau_{ss}(s)$. 
Moreover, the unfolding $G$ is versal at this point. 
\end{enumerate}
\end{thm}

\subsection{Singular set of the generating family}

Differentiating \eqref{definet} with respect to $s$ we obtain
\begin{equation}\label{derivativet}
-\left[ \frac{\partial}{\partial s}x(s,t), c(s,t)\right]+\left[ x-x(s,t),  \frac{\partial}{\partial s}c(s,t)\right]=0.
\end{equation}
Now 
\begin{equation}\label{derivativeg}
G_s(s,x)=g_s+g_t t_s=\left[ x_s+x_tt_s, c \right]=\left[ \frac{\partial}{\partial s}x(s,t), c(s,t) \right],
\end{equation}
and so \eqref{derivativet} implies that
\begin{equation}\label{derivativeg2}
G_s(s,x)=\left[ x-x(s,t),  \frac{\partial}{\partial s}c(s,t)  \right].
\end{equation}

\begin{proposition}\label{prop:Gs}
$G_s(s,x)=0$ if and only if $x=x(s,t)$.
\end{proposition}
\begin{proof}
Denote $M=\{(s,x)\in I\times U|\ x=x(s,t(s,x))\}$ and $N=\{(s,x)\in I\times U|\ G_s(s,x)=0\}$. Then $M$ and $N$ are smooth surfaces and, from \eqref{derivativeg2},  $M\subset N$. 
By reducing the neighborhood if necessary, we can conclude that $M=N$. 
\end{proof}




\begin{lemma}
For any $(s,x)$, 
\begin{equation}\label{calculadelta}
\left[ \frac{\partial}{\partial s}x(s,t), \frac{\partial}{\partial s}c(s,t) \right]=-(1+t_s^2)\delta(s,t).
\end{equation} 
\end{lemma}
\begin{proof}
We have that $ \frac{\partial}{\partial s}x(s,t)=x_s+t_s x_t$ and $\frac{\partial}{\partial s}c(s,t)=c_s+t_s c_t=-x_t+t_s x_s$. Now 
\eqref{calculadelta} follows easily. 
\end{proof}

Differentiating \eqref{derivativet} we obtain
\begin{equation}\label{derivative2t}
-\left[ \frac{\partial^2}{\partial s^2}x(s,t), c(s,t) \right] -2\left[ \frac{\partial}{\partial s}x(s,t), \frac{\partial}{\partial s}c(s,t)\right]+\left[ x-x(s,t), \frac{\partial^2}{\partial s^2}c(s,t)    \right]=0.
\end{equation}
If $x(s,t)=x$, then
\begin{equation}\label{derivative2t0}
\left[ \frac{\partial^2}{\partial s^2}x(s,t), c(s,t)\right] =2(1+t_s^2)\delta(s,t). 
\end{equation}

\begin{proposition}\label{prop:GsGss}
$G_s=G_{ss}=0$ if and only if $\delta(s,t)=0$. 
\end{proposition}

\begin{proof}
Differentiating \eqref{derivativeg} we obtain
\begin{equation}\label{derivative2g}
G_{ss}=\left[ \frac{\partial^2}{\partial s^2}x(s,t), c(s,t)\right]+\left[ \frac{\partial}{\partial s}x(s,t), \frac{\partial}{\partial s}c(s,t) \right].
\end{equation}
At $x(s,t)=x$, using \eqref{calculadelta} and \eqref{derivative2t0}, we obtain  
$$
G_{ss}=(t_s^2+1)\delta(s,t). 
$$
Thus $G_{ss}=0$ if and only if $\delta(s,t)=0$.
\end{proof}

We conclude from propositions \ref{prop:Gs} and \ref{prop:GsGss} that theorem \ref{thm:familyG}(1) holds. 

\subsection{Bifurcation set of the generating family}

\begin{lemma}
At $x=x(s,t)$, $\delta(s,t)=0$, 
\begin{equation}\label{xszero}
\frac{\partial}{\partial s}x(s,t)=x_s+x_tt_s=0.
\end{equation} 
In other words, $\alpha(s)=t_s$, where $\alpha(s)$ is defined by \eqref{definealpha}. 
\end{lemma}
\begin{proof}
At $x=x(s,t)$, $\delta(s,t)=0$ we have $x_s+\alpha x_t=0$. Using now \eqref{derivativet} we obtain that
$$
\left[-\alpha x_t+t_s x_t, c \right]=0.
$$
Thus $\alpha=t_s$, proving the lemma. 
\end{proof}

\begin{proposition}\label{versalA2}
The family $G(s,x)$ is versal at an $A_2$ point. 
\end{proposition}
\begin{proof}
We first calculate $(G_{x_1},G_{x_2})=g_t(t_{x_1},t_{x_2})$. Differentiating \eqref{definet} we obtain
\[
\left\{
\begin{array}{c}
(1-(x_1)_t t_{x_1})C=(x_2)_t t_{x_1}D\\
(x_1)_t t_{x_2} C= (1-(x_2)_t t_{x_2})D
\end{array}
\right.
\]
and thus $(t_{x_1},t_{x_2})=\frac{1}{[x_t,c]}(C,D)$. We conclude that $G_x=(C,D)$. 

Now, using \eqref{xszero}, we conclude that
$$
\frac{\partial}{\partial s}c(s,t(s,x))=c_s+c_t t_s=-x_t+x_s t_s=-(1+t_s^2)x_t\neq 0,
$$
which proves the versality of $G$ at an $A_2$ point (see \cite{Giblin92}, page 149).  
\end{proof}

Differentiating \eqref{derivative2t} we obtain
\begin{equation}\label{derivative3t}
-\left[ \frac{\partial^3}{\partial s^3}x, c \right] -3\left[ \frac{\partial^2}{\partial s^2}x, \frac{\partial}{\partial s}c\right]-3\left[ \frac{\partial}{\partial s}x, \frac{\partial^2}{\partial s^2}c\right]+\left[ x-x(s,t), \frac{\partial^3}{\partial s^3}c   \right]=0.
\end{equation}

But differentiating \eqref{calculadelta} and taking $x=x(s,t)$, $\delta(s,t)=0$ we obtain
\begin{equation}\label{calcula2delta}
\left[ \frac{\partial^2}{\partial s^2}x(s,t),  \frac{\partial}{\partial s}c(s,t) \right]=-(t_s^2+1)\frac{\partial}{\partial s}\delta(s,t),
\end{equation}
We conclude that
\begin{equation}\label{derivative3t0}
\left[ \frac{\partial^3}{\partial s^3}x(s,t), c(s,t) \right] =3(t_s^2+1)\frac{\partial}{\partial s}\delta(s,t).
\end{equation}

\begin{proposition}\label{GsGssGsss}
$G_s=0,\ G_{ss}=0,\ G_{sss}=0$ if and only if $x(s,t)=x$, $\delta(s,t)=0$ and $\frac{\partial}{\partial s}\delta(s,t)=0$.
\end{proposition}
\begin{proof}
Differentiating \eqref{derivative2g} we obtain 
\begin{equation}\label{derivative3g}
G_{sss}=\left[ \frac{\partial^3}{\partial s^3}x(s,t), c(s,t)\right]+2\left[ \frac{\partial^2}{\partial s^2}x(s,t), \frac{\partial}{\partial s}c(s,t) \right]+\left[ \frac{\partial}{\partial s}x(s,t), \frac{\partial^2}{\partial s^2}c(s,t) \right].
\end{equation}
At $\delta=0$ we have
$$
G_{sss}= (t_s^2+1)\frac{\partial}{\partial s}\delta(s,t).
$$
Thus $G_{sss}=0$ if and only if $\frac{\partial}{\partial s}\delta=\delta_s+\delta_tt_s=0$. 
\end{proof}

 We conclude from propositions \ref{prop:GsGss}, \ref{versalA2} and \ref{GsGssGsss} that theorem \ref{thm:familyG}(2) holds. We remark that 
 $\frac{\partial}{\partial s}\delta=0$ is equivalent to $\alpha(s)=\tau_s(s)$, where $\alpha$ and $\tau$ were defined before theorem \ref{thm:singularplanar}. Hence these propositions
 also gives another proof of  theorem \ref{thm:singularplanar}(1). 
 
 \subsection{  $A_3$ points of the generating family}

\begin{proposition}\label{versalA3}
The family $G(s,x)$ is versal at an $A_3$ point.
\end{proposition}
\begin{proof}
In proposition \ref{versalA2}, we have calculated $G_x=(C,D)$ and $(G_x)_s=-(1+t_s^2)x_t$. In order to prove the versality of $G(s,x)$ at an $A_3$ point we must calculate
$(G_x)_{ss}$ and check that $[(G_x)_s,(G_x)_{ss}]\neq 0$ (see \cite{Giblin92}, page 149). This is equivalent to proving that 
$$
\left[ x_t, \frac{\partial}{\partial s}x_t \right]=\left[ x_t,x_{st}+t_sx_{tt} \right] \neq 0.
$$ 
But we have assumed that $\delta_t\neq 0$. Using \eqref{xszero} we have
$$
\delta_t=\left[ x_{st},x_t \right]+ \left[ x_s,x_{tt} \right]=- \left[ x_t, x_{st}+t_sx_{tt} \right],
$$
thus completing the proof of the proposition.
\end{proof}

\begin{lemma} 
If $x=x(s,t)$ and $\delta=\frac{\partial}{\partial s}\delta=0$, then
\begin{equation}\label{xsszero}
\frac{\partial^2}{\partial s^2}x(s,t)=0. 
\end{equation}
This is equivalent to saying that $t_{ss}=\alpha'(s)$. 
\end{lemma}
\begin{proof}
Since $\delta=\left[ x_s, x_t \right]$, we have 
$$
\frac{\partial}{\partial s}\delta=\left[ x_{ss}+x_{st}t_s, x_t \right] + \left[ x_s, x_{st}+x_{tt}t_s \right].
$$
At $x=x(s,t)$, $\delta=0$,  \eqref{xszero} implies that
$$
\frac{\partial}{\partial s}\delta=\left[ x_{ss}+2x_{st}t_s+x_{tt}t_s^2, x_t \right].
$$
Thus, if $x=x(s,t)$, $\delta=\frac{\partial}{\partial s}\delta=0$, 
$x_{ss}+2x_{st}t_s+x_{tt}t_s^2=\beta x_t$, for some $\beta\in\R$. On the other hand, by \eqref{derivative2t0} we obtain
$$
\left[ (\beta+t_{ss})x_t, c \right]=0,
$$
which implies that in fact $\beta=-t_{ss}$, thus proving the lemma.
\end{proof}

Differentiating \eqref{derivative3t} and taking $x=x(s,t)$, $\delta=\frac{\partial}{\partial s}\delta=0$ we obtain 
\begin{equation*}\label{derivative4t}
\left[ \frac{\partial^4}{\partial s^4}x(s,t), c(s,t) \right] +4\left[ \frac{\partial^3}{\partial s^3}x(s,t),  \frac{\partial}{\partial s}c(s,t) \right]=0.
\end{equation*} 
But differentiating  \eqref{calcula2delta} and taking $x=x(s,t)$, $\delta=\frac{\partial}{\partial s}\delta=0$ we get
\begin{equation}\label{calcula3delta}
\left[ \frac{\partial^3}{\partial s^3}x(s,t),  \frac{\partial}{\partial s}c(s,t) \right]=-(t_s^2+1)\frac{\partial^2}{\partial s^2}\delta(s,t).
\end{equation}
Thus
\begin{equation}\label{derivative4t0}
\left[ \frac{\partial^4}{\partial s^4}x(s,t), c(s,t) \right]=4(t_s^2+1)\frac{\partial^2}{\partial s^2}\delta(s,t).
\end{equation}

\begin{proposition}\label{GsGssGsssGssss}
$G_s=0,\ G_{ss}=0,\ G_{sss}=0,\ G_{ssss}\neq 0$ if and only if $x(s,t)=x$, $\delta(s,t)=0$, $\frac{\partial}{\partial s}\delta(s,t)=0$ and $\frac{\partial^2}{\partial s^2}\delta(s,t)\neq0$.
This last condition is equivalent to $\tau_{ss}\neq t_{ss}=\alpha'(s)$. 
\end{proposition}
\begin{proof}
Differentiating \eqref{derivative3g}  we obtain 
\begin{equation}\label{derivative4g}
G_{ssss}=\left[ \frac{\partial^4}{\partial s^4}x, c \right]+3\left[ \frac{\partial^3}{\partial s^3}x, \frac{\partial}{\partial s}c \right]
+3\left[ \frac{\partial^2}{\partial s^2}x, \frac{\partial^2}{\partial s^2}c \right]+\left[ \frac{\partial}{\partial s}x, \frac{\partial^3}{\partial s^3}c \right].
\end{equation}
Taking $x=x(s,t)$, $\delta=0$ and $\frac{\partial}{\partial s}\delta=0$, we obtain from \eqref{xsszero}, \eqref{calcula3delta} and \eqref{derivative4t0}  that
\begin{equation*}
G_{ssss}=(t_s^2+1)\frac{\partial^2}{\partial s^2}\delta(s,t).
\end{equation*}
\end{proof}

From propositions \ref{GsGssGsss} and \ref{GsGssGsssGssss} we obtain theorem \ref{thm:familyG}(3). Also, these propositions give another proof of theorem \ref{thm:singularplanar}(2). 

\section{Singularities of the improper affine map}

Consider the family $\tilde{G}:I\times U\times\R\to\R$ defined by
$$
\tilde{G}(s,x,z)=G(s,x)-z.
$$

\begin{lemma}
The discriminant set of $\tilde{G}$ coincides with the image of the improper affine map. 
\end{lemma}
\begin{proof}
If $\tilde{G}=\tilde{G_s}=0$, then, by theorem \ref{thm:familyG}(1), $x=x(s,t)$. Thus $z=G(s,x(s,t))=g(s,t)$. 
\end{proof}

The following theorem is a direct consequence of theorem \ref{thm:familyG}:

\begin{thm}\label{thm:familyG1}
Take $(s,x,z)\in I\times U\times \R$:
\begin{enumerate}
\item
$(s,x,z)$ is an $A_2$ point for $\tilde{G}$  if and only if $x=x(s,t)$, $\delta(s,t(s,x))=0$ and $\frac{\partial}{\partial s}\delta(s,t(s,x))\neq 0$, or equivalently, 
$\alpha(s)\neq\tau_s(s)$. Moreover, the unfolding $\tilde{G}$ is versal at this point. 
\item
$(s,x,z)$ is an $A_3$ point for $\tilde{G}$  if and only if $x=x(s,t)$, $\delta(s,t(s,x))=0$, $\frac{\partial}{\partial s}\delta(s,t(s,x))=0$ and $\frac{\partial^2}{\partial s^2}\delta(s,t(s,x))\neq 0$, or equivalently, $\alpha(s)=\tau_s(s)$ and $\alpha'(s)\neq\tau_{ss}(s)$. 
Moreover, the unfolding $\tilde{G}$ is versal at this point. 
\end{enumerate}
\end{thm}

\begin{corollary}\label{prop:singAS}
\begin{enumerate}
\item
If $\delta(s,t)=0$ but $\frac{\partial}{\partial s}\delta(s,t)\neq 0$, then $q(S)$ is a cuspidal edge at $q(s,t)$.
\item
If $\delta(s,t)=\frac{\partial}{\partial s}\delta(s,t)=0$ but $\frac{\partial^2}{\partial s^2}\delta(s,t)\neq 0$, then $q(S)$ is a swallowtail at $q(s,t)$.
\end{enumerate}
\end{corollary}
\begin{proof}
In case (1), theorem \ref{thm:familyG1}(1) says that $\tilde{G}$ has an $A_2$ point and $\tilde{G}$ is versal at this point, so the discriminant set is a cuspidal edge. 
In case (2), theorem \ref{thm:familyG1}(2) says that $\tilde{G}$ has an $A_3$ point and $\tilde{G}$ is versal at this point, so the discriminant set is a swallowtail. 
\end{proof}

\begin{example}\label{example3}
Consider the improper affine sphere $q=(x,g)$ of example \ref{example1}. As shown in example \ref{example2}, the image of $S=\{(s,t)|\ s^2+t^2=1\}$ by the planar map $x$ is regular except at $(1,0)$, 
$(-\frac{1}{2},\frac{\sqrt{3}}{2})$ and $(-\frac{1}{2},-\frac{\sqrt{3}}{2})$, where there are ordinary cusps. By proposition \ref{prop:singAS}, we have that the image of $S$ by the immersion $q$ are cuspidal edges
at all points except the above three, where they are swallowtails (see figure \ref{swallow}). 

\begin{figure}[htb]
 \centering
 \includegraphics[width=0.30\linewidth]{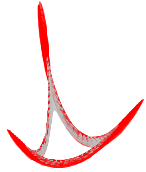}
 \caption{The image of the improper affine map of example \ref{example3} close to the singularity set. Observe
 the three swallowtails at points whose projections on the plane are the cusps of figure \ref{trescuspides}.  }
\label{swallow}
\end{figure}

\end{example}

\begin{example}
In section 6 of \cite{Galvez07}, a class of convex improper affine spheres is constructed from its singularities. In these examples,
the singularity set $S$ is defined by $t=0$ and thus $\tau(s)=0$. Moreover, $x_s\neq 0$. Thus we can define ${\tilde\alpha}(s)$ by the relation
${\tilde\alpha}(s)x_s+x_t=0$ at $t=0$, which is a variation of \eqref{definealpha}. In this case, the condition for $x(s,t)$ to be a regular point of $x(S)$ is ${\tilde\alpha}(s)\tau_s(s)\neq 1$, which is always 
satisfied. By corollary \ref{prop:singAS}(1), the singularity of the improper affine map is a cuspidal edge at each point $(s,0)$. One specific example of this construction, extracted from \cite{Galvez07}, 
is given by
\[
\begin{array}{cc}
x=&\left(\cos(s)\cosh(t)+\frac{3}{5}\cos(s)\sinh(t)-\frac{1}{5}\cos(3s)\sinh(3t), \right. \\
&\ \ \left. \sin(s)\cosh(t)-\frac{3}{5}\sin(s)\sinh(t)-\frac{1}{5}\sin(3s)\sinh(3t)\right),
\end{array}
\]
\[
\begin{array}{ll}
g=\frac{1}{100} &\left( \ 62t+10\cos(2s)(3-2\cosh(2t)+\cosh(4t)) \right.\\
 &\left. -24\cos(4s)\cosh^3(t)\sinh(t)-16\sinh(2t)-\sinh(6t) \right).
\end{array}
\]
\end{example}

\end{document}